\documentclass[reqno,12pt]{amsart}
\usepackage[utf8]{inputenc}
\usepackage[T1]{fontenc}
\usepackage{amsmath,amsfonts,amssymb,amsthm}
\usepackage{extarrows}
\usepackage{geometry}
\usepackage[normalem]{ulem}
\usepackage{graphicx}
\graphicspath{{./images/}}
\usepackage{eso-pic}
\usepackage{url}
\usepackage[all]{xy}
\usepackage{esint}
\usepackage[pdftex,
            colorlinks=true,
            linkcolor=blue,
            citecolor=red,
            pdfauthor={Alexis de Villeroché},
            pdftitle={},
            pdfcreator={pdflatex}]{hyperref}
\usepackage{bm}
\usepackage{dsfont}
\definecolor{verde}{RGB}{20,150,100}

\newcommand{\Z}{\mathbb{Z}}

\newcommand{\R}{\mathbb{R}}

\newcommand{\Rd}{\R^d}

\renewcommand{\O}{\Omega}

\newcommand{\A}{\mathcal{A}}
\newcommand{\F}{\mathcal{F}}
\newcommand{\Fpq}{\mathcal{F}_{p,q}}

\renewcommand{\phi}{\varphi}

\newcommand{\cp}{\mathop{\rm cap}\nolimits}
\newcommand{\dive}{\mathop{\rm div}\nolimits}
\newcommand{\eps}{\varepsilon}

\newcommand{\be}{\begin{equation}}
\newcommand{\ee}{\end{equation}}
\newcommand{\bib}[4]{\bibitem{#1}{\sc#2: }{\it#3. }{#4.}}

\numberwithin{equation}{section}
\theoremstyle{plain}
\newtheorem{trm}{Theorem}[section]

\newtheorem{prop}[trm]{Proposition}
\newtheorem{df}[trm]{Definition}

\theoremstyle{remark}
\newtheorem{rem}[trm]{\bf Remark}

\title[Optimal domains for the Cheeger inequality]{Optimal domains for the Cheeger inequality}
\author{Dorin Bucur, Giuseppe Buttazzo, Alexis de Villeroch\'e}
\date{}

\begin{document}

\begin{abstract}
In this paper we consider the scale invariant shape functional
$$\F_{p,q}(\O)=\frac{\lambda_p^{1/p}(\O)}{\lambda_q^{1/q}(\O)},$$
where $1\le q<p\le+\infty$ and $\lambda_p(\O)$ (respectively $\lambda_q(\O)$) is the first eigenvalue of the $p$-Laplacian $-\Delta_p$ (respectively $-\Delta_q$) with Dirichlet boundary condition on $\partial\O$. We study both the maximization and minimization problems for $\F_{p,q}$, and show the existence of optimal domains in $\R^d$, along with some of their qualitative properties. Surprisingly, the case of a bounded box $D$ constraint
$$\max\Big\{\lambda_q(\O)\ :\ \O\subset D,\ \lambda_p(\O)=1\Big\},$$
leads to a problem of different nature, for which the existence of a solution is shown by analyzing optimal capacitary measures. In the last section we list some interesting questions that, in our opinion, deserve to be investigated.
\end{abstract}

\maketitle



\section{Introduction}\label{Intro}

The point of departure for this study is the celebrated Cheeger inequality \cite{C71}, which provides a fundamental link between spectral and geometric properties of domains. Specifically, for any open, bounded subset $\O\subset\R^d$, the inequality asserts that
\be\label{ch21}
\frac{\sqrt{\lambda(\O)}}{h(\O)}\ge\frac12
\ee
where $\lambda(\O)$ denotes the first eigenvalue of the Dirichlet Laplacian $-\Delta$ on $\O$, and $h(\O)$ is the Cheeger constant associated with the domain.

The first Dirichlet eigenvalue $\lambda(\O)$ is defined by the Rayleigh quotient as
\[\begin{split}
\lambda(\O)&=\inf\bigg\{\frac{\int|\nabla u|^2\,dx}{\int|u|^2\,dx}\ :\ u\in C^1_c(\O)\setminus\{0\}\bigg\}\\
&=\min\bigg\{\frac{\int|\nabla u|^2\,dx}{\int|u|^2\,dx}\ :\ u\in H^1_0(\O)\setminus\{0\}\bigg\},
\end{split}\]
where all integrals are understood over the entire space $\R^d$, with the convention that functions in $H^1_0(\O)$ are considered as extended by zero outside the domain $\O$.

The Cheeger constant $h(\O)$, a geometric quantity capturing isoperimetric properties of the domain, is defined as
$$h(\O)=\inf\bigg\{\frac{P(E)}{|E|}\ :\ E\Subset\O\bigg\}.$$
where $P(E)$ denotes the perimeter of the set $E$, and $|E|$ its Lebesgue measure.

The inequality \eqref{ch21} highlights a deep connection between the geometry of a domain and the spectral behavior of the Laplacian defined on it, forming the foundation for much of the analysis that follows.

There exist other equivalent formulations of the Cheeger constant $h(\O)$, each offering a distinct perspective on its variational and geometric character. A classical expression, mirroring the Rayleigh quotient but in the $L^1$ setting, is given by
\[\begin{split}
h(\O)&=\inf\bigg\{\frac{\int|\nabla u|\,dx}{\int|u|\,dx}\ :\ u\in C^1_c(\O)\setminus\{0\}\bigg\}\\
&=\inf\bigg\{\frac{\int|\nabla u|\,dx}{\int|u|\,dx}\ :\ u\in W^{1,1}_0(\O)\setminus\{0\}\bigg\},
\end{split}\]
where the infimum is taken over nontrivial functions in either the space of compactly supported smooth functions or in the Sobolev space $W^{1,1}_0(\O)$.

When the domain $\O$ possesses a Lipschitz boundary, the Cheeger constant $h(\O)$ can be characterized as
$$h(\O)=\inf\bigg\{\frac{P(E)}{|E|}\ :\ E\subset\O\bigg\},$$
thus highlighting the dual analytic and geometric nature of the Cheeger constant. Interestingly, the Cheeger constant also emerges as the limit case of the first eigenvalue of the $p$-Laplacian as $p\to1$. For $1\le p<+\infty$ the first Dirichlet eigenvalue of the $p$-Laplacian is defined by\be\label{lp}
\lambda_p(\O)=\inf\bigg\{\frac{\int|\nabla u|^p\,dx}{\int|u|^p\,dx}\ :\ u\in W^{1,p}_0(\O)\setminus\{0\}\bigg\},
\ee
so that, in this framework, the Cheeger constant satisfies $h(\O)=\lambda_1(\O)$, emphasizing its role as the $p=1$ analog of the principal eigenvalue. For a comprehensive account of the Cheeger constant and related optimization problems, we refer the reader to \cite{F21}, \cite{L15}, \cite{P11}, \cite{P17}, and \cite{PS25}.

Moreover, the normalized quantity $\lambda_p^{1/p}(\O)$ admits a meaningful extension to the case $p=\infty$, owing to the well-known asymptotic result (see, for instance, \cite{JLM99}):
$$\lim_{p\to\infty}\lambda_p^{1/p}(\O)=\rho(\O),$$
where $\rho(\O)$ denotes the so-called {\it inradius} of the domain $\O$, defined as the largest radius of a ball that can be inscribed in $\O$. Equivalently, $\rho(\O)$ is the maximum of the distance function to the boundary $\partial\O$, thereby tying together spectral asymptotics and intrinsic geometry in a particularly elegant way.

A broader perspective on Cheeger-type inequalities can be obtained by introducing the scale invariant shape functional
\be\label{fpq}
\F_{p,q}(\O)=\frac{\lambda_p^{1/p}(\O)}{\lambda_q^{1/q}(\O)},
\ee
which involves the first eigenvalues of the Dirichlet $p$- and $q$-Laplacians, with $p>q$. Within this framework, the classical Cheeger inequality \eqref{ch21} appears as a special instance of a more general principle. Specifically, for any $1\le q\le p\le+\infty$ and for every bounded open subset $\O$ of $\R^d$, the following inequality holds:
\be\label{cheegerp}
\F_{p,q}(\O)\ge\frac{q}{p}\qquad\hbox{for every }1\le q\le p\le+\infty.
\ee
This inequality can also be reformulated in terms of a monotonicity property, as:
$$\text{\it the map $p\mapsto p\lambda_p^{1/p}(\O)$ is monotonically nondecreasing.}$$
The proof of this family of inequalities is relatively elementary and rests on a careful application of H\"older's inequality; a detailed exposition can be found in \cite{BBP22}. While the constant $q/p$ in \eqref{cheegerp} is not optimal in general, it is known to be asymptotically sharp in high dimensions, that is, as the ambient space dimension $d\to\infty$ (see \cite{BBP22} for precise asymptotic estimates).

This variational comparison not only unifies classical spectral bounds under a single conceptual framework, but also provides insight into the interplay between the domain's geometry and the behavior of eigenvalues across the full range of $p$-Laplacian operators. An analysis of the functional $\F_{p,q}$ on manifolds can be found in \cite{CM03}.

It is worth observing that when the domain $\O$ has finite Lebesgue measure, the infimum in the variational characterization of $\lambda_p(\O)$ given in \eqref{lp} is indeed attained. In this case, there exists a generalized eigenfunction $u_p$ that realizes the minimum and satisfies the associated quasilinear eigenvalue PDE
$$\begin{cases}
-\Delta_p u_p=\lambda_p(\O) |u_p|^{p-2}\,u_p,\\
u_p\in W^{1,p}_0(\O),\quad\|u_p\|_{L^p(\O)}=1.
\end{cases}$$
where $\Delta_p u:=\dive\big(|\nabla u|^{p-2}\nabla u\big)$ denotes the $p$-Laplacian operator. It is convenient to extend the definition of the eigenvalue $\lambda_p(\O)$ to unbounded domains; namely, one defines
$$\lambda_p(\O)=\inf\big\{\lambda_p(A)\ :\ A\subset\O,\ A\text{ bounded}\big\}.$$
This construction preserves the variational character of the eigenvalue while ensuring that the definition remains meaningful even when $\O$ is not of finite measure.

With this convention, the shape functional $\F_{p,q}(\O)$ introduced in \eqref{fpq} remains well-defined for all open sets $\O\subset\R^d$ (bounded or not) such that $\lambda_q(\O)\ne0$. This generality is particularly useful when studying shape optimization problems related to the function $\F_{p,q}$, as we will do in the following sections.

Our primary interest lies in studying the extremal behavior - both minimization and maximization - of the shape functional $\F_{p,q}$. To this end, we define the associated optimal bounds:
\be\label{df:mpq,Mpq}
\begin{array}{l}
m(p,q)=\inf\big\{\Fpq(\O)\ :\ \O\subset\Rd,\ \O\mbox{ open, }\lambda_q(\O)>0\big\},\\
M(p,q) = \sup\big\{ \Fpq(\O)\ :\ \O\subset\Rd,\ \O\mbox{ open, }\lambda_q(\O)>0\big\}.
\end{array}
\ee
A key feature of the shape functional $\F_{p,q}$ is its invariance under scaling transformations. Indeed, using the well-known homogeneity of the $p$-Laplacian eigenvalue under homotheties,
$$\lambda_p(t\O)=t^{-p}\lambda_p(\O)\qquad\text{for every }t>0,$$
it follows immediately that the functional $\F_{p,q}$ satisfies
$$\F_{p,q}(t\O)=\F_{p,q}(\O)\qquad\text{for every }t>0.$$
This scaling invariance allows one to reduce the study of $\F_{p,q}$ to families of domains normalized with respect to either $\lambda_p$ or $\lambda_q$. More precisely, we may equivalently express the extremal values as:
$$\begin{array}{l}
m(p,q)=\Big(\inf\big\{\lambda_p(\O)\ :\ \lambda_q(\O)=1\big\}\Big)^{1/p}=\Big(\sup\big\{\lambda_q(\O)\ :\ \lambda_p(\O)=1\big\}\Big)^{-1/q}\\
M(p,q)=\Big(\sup\big\{\lambda_p(\O)\ :\ \lambda_q(\O)=1\big\}\Big)^{1/p}=\Big(\inf\big\{\lambda_q(\O)\ :\ \lambda_p(\O)=1\big\}\Big)^{-1/q}.
\end{array}$$

We now state the existence result of optimal domains for the functional $\F_{p,q}$.

\begin{trm}\label{Th:ExistenceMpqmpq}
Let $1\le q\le p\le\infty$. Then there exist two open and smooth subsets of $\R^d$, denoted by $\O_{p,q}^m$ and $\O_{p,q}^M$, such that the infimum and supremum in \eqref{df:mpq,Mpq} are attained:
$$\Fpq(\O_{p,q}^m)=m(p,q)\qquad\text{and}\qquad\Fpq(\O_{p,q}^M)=M(p,q).$$
Moreover, the quantity $M(p,q)$ is finite if and only if $q>d$.
\end{trm}

\begin{rem}
It is worth emphasizing that the optimal domains $\O_{p,q}^m$ and $\O_{p,q}^M$ obtained in Theorem \ref{Th:ExistenceMpqmpq} may, in general, be unbounded. As a consequence, one cannot always associate a well-defined eigenfunction to $\lambda_p(\O)$ or $\lambda_q(\O)$, since the corresponding eigenvalues may not be attained by any function in the associated Sobolev space.

In several spectral optimization problems, the question of whether optimal shapes exist among all admissible domains in the Euclidean space $\R^d$ is both subtle and technically challenging (see for example \cite{B12}, \cite{MP13}). By contrast, in the present setting, the existence of an optimal domain in the entire space $\R^d$ can be established through a relatively straightforward argument. However, the situation becomes considerably more intricate when additional space constraints are imposed, specifically, when the admissible domains are required to lie within a fixed bounded set, that is $\O\subset D$ with $D$ bounded. In such cases, a much more refined analysis is necessary, involving the use of capacitary measures, as will be discussed in detail in Section \ref{s:DDD}.

In addition, the non-uniqueness of optimal sets should be noted. Due to the definition of the functional $\F_{p,q}$, optimal domains can be locally modified - e.g., by removing or adding compact subsets - without affecting the value of the functional. Thus, the minimizers and maximizers should be interpreted as equivalence classes of domains with the same spectral properties. \end{rem}

Among the possible optimal domains $\O_{p,q}^m$ and $\O_{p,q}^M$ achieving the values $m(p,q)$ and $M(p,q)$, a particularly interesting class arises when $p>q>d$. In this regime, we are able to construct optimal domains that are complements of discrete sets in $\R^d$. This phenomenon is closely related to the theory of Sobolev capacities: when $p>q>d$, single points in $\R^d$ have positive $W^{1,q}$-capacity, and therefore their removal can influence the spectral quantities under consideration.

We formalize this observation in the following result.

\begin{trm}\label{th:discret}
Let $d<q<p\le +\infty$. Then there exist open subsets $\O_{p,q}^m$ and $\O_{p,q}^M$ of $\R^d$ such that
$$\Fpq(\O_{p,q}^m)=m(p,q)\qquad\text{and}\qquad\Fpq(\O_{p,q}^M)=M(p,q)$$ 
and moreover, the complements $\Rd\setminus\O_{p,q}^m$ and $\Rd\setminus\O_{p,q}^M$ are discrete sets with no accumulation points.
\end{trm}

A particularly striking case arises when $p=\infty$, in which the functional $\F_{\infty,q}$ takes the explicit form
$$\F_{\infty,q}(\O)=\frac{1}{\rho(\O)\,\lambda_q^{1/q}(\O)}.$$
where $\rho(\O)$ denotes the inradius of $\O$.

By the scaling invariance of both $\rho(\O)$ and $\lambda_q(\O)$, we may, without loss of generality, impose the normalization $\rho(\O)=1$. Under this constraint, the optimization of $\F_{\infty,q}(\O)$ reduces to the study of the eigenvalue $\lambda_q(\O)$ among all domains with fixed inradius.

The maximization problem for $\lambda_q$ in this setting is straightforward: it is well known that among all open sets with fixed inradius, the Euclidean ball maximizes $\lambda_q$, thereby achieving the minimum of $\F_{\infty,q}$.

In contrast, the minimization problem for $\lambda_q$ under the constraint that $\rho(\O)=1$ is much richer and more intricate. In the two-dimensional case ($d=2$), we conjecture that an optimal configuration is given by the domain $\O^*=\R^2\setminus X$
where $X$ is a discrete set of points forming the centers of a regular hexagonal tessellation of the plane. This conjecture is supported by numerical simulations and geometric considerations, which reduce the question to an optimization problem involving triangles inscribed in a unit circle. In this simplified setting, the hexagonal arrangement emerges naturally as the most efficient configuration in minimizing $\lambda_q$ under the inradius constraint.

When an additional geometric constraint is imposed on the admissible domains - namely, requiring that $\O$ be contained in a fixed bounded box $D\subset\R^d$ - the analysis of the optimization problems becomes significantly more subtle. In this constrained setting, the existence of optimal domains requires the theory of Sobolev capacities and work within the framework of $p$-quasi open sets. We address this issue in detail in Section~\ref{s:DDD}, where we establish the following existence result.

\begin{trm}\label{quasi}
Let $1<q<p <+\infty$. Then, for every bounded open set $D\subset\R^d$, there exist optimal $p$-quasi open sets $\O_1$ and $\O_2$ solving the following constrained variational problems:
\be\label{pb:boundedmin}
\min\big\{\lambda_p(\O)\ :\ \O\subset D,\ \lambda_q(\O)=1\big\};
\ee
\be\label{pb:boundedmax}
\max\big\{\lambda_q(\O)\ :\ \O\subset D,\ \lambda_p(\O)=1\big\}.
\ee
respectively.
\end{trm}

\begin{rem}
It is important to observe that, in the bounded setting, the scale-invariant formulation of the functional $\F_{p,q}$ no longer yields an equivalent problem. Indeed, the constraint $\O\subset D$ prohibits the use of homotheties, and therefore the scale-invariance, that plays a crucial role in the unbounded case, is lost.

As a consequence, the constrained problem
$$
\inf\big\{\F_{p,q}(\O)\ :\ \O\subset D\big\}
$$
is not equivalent to either of the problem \eqref{pb:boundedmin} or \eqref{pb:boundedmax}, and must be treated using different techniques.
\end{rem}

The structure of the paper is as follows. In Section~\ref{sprel}, we collect several preliminary results and technical tools that will be employed throughout the analysis. Section~\ref{sexis} is devoted to proving the existence of optimal domains for the unconstrained functional $\F_{p,q}$. In Section~\ref{s:DDD}, we focus on the constrained case where the admissible domains are required to lie within a fixed bounded region $D$. Finally, in Section~\ref{s:open}, we conclude by presenting a selection of open problems and directions for future investigation.

\section{Preliminary results}\label{sprel}

In this section, we introduce the principal concepts that will be used in the rest of the paper. Alongside these definitions, we establish a few preliminary results that will be crucial in the subsequent analysis. For a general overview of shape optimization problems and further related details, we refer the reader to the monographs \cite{BB05} and \cite{HP05}.

\begin{prop}\label{prop:minimalityHolder}
Let $\O\subset\R^d$ be an open set, and let $q\le p$. Then the following inequality holds:
$$\Fpq(\O)\ge\frac{q}{p}.$$
\end{prop}

\begin{proof}
This result follows as a direct consequence of H\"older's inequality. \\
Let $\eps>0$ and choose a function $v\in W^{1,p}_0(\O)$ such that
$$\lambda_p(\O)\leq \frac{\int_\O|\nabla v|^pdx}{\int_\O|v|^pdx}\le\big(\lambda_p(\O)^{q/p}+\eps\big)^{p/q}.$$
Now consider the function $v^{p/q}\in W^{1,q}_0(\O)$. A direct computation yields:
$$\begin{array}{ll}
\lambda_q(\O)&\le\displaystyle\frac{\int_\O|\nabla v^{p/q}|^qdx}{\int_\O|v^{p/q}|^qdx}\\
&=\left(\frac{p}{q}\right)^q\,\displaystyle\frac{\int_\O|v|^{p-q}|\nabla v|^qdx}{\int_\O|v|^pdx} \\
&\le\left(\frac{p}{q}\right)^q\,\displaystyle\frac{\left(\int_\O|v|^pdx\right)^{(p-q)/p}\,\left(\int_\O|\nabla v|^pdx\right)^{q/p}}{\int_\O|v|^pdx}\\
&\le\left(\frac{p}{q}\right)^q\,\big(\lambda_p(\O)^{q/p}+\eps\big).
\end{array}
$$
Letting $\eps\to0$ we obtain the desired inequality.
\end{proof}

\begin{trm}\label{th:BozzolaBrasco24}
Let $d<p<+\infty$. Then there exists a constant $C_{p,d}>0$, depending only on $p$ and $d$, such that for every open set $\O\subset\R^d$ the following two-sided estimate holds:
$$C_{p,d}\le\rho(\O)^p\,\lambda_p(\O)\le\lambda_p(B(0,1))$$
where $\rho(\O)$ denotes the inradius of $\O$.
\end{trm}

\begin{proof}
The proof of this result can be found in \cite{BB24}.
\end{proof}

We recall that for a Sobolev function $u\in W^{1,p}(\R^d)$, with $p>1$, the {\it quasi-continuous} representative
$$\tilde u(x)=\lim_{\eps\to0}\frac{1}{|B(x,\eps)|}\int_{B(x,\eps)}\!\!\!u(y)\,dy$$
is defined (in the sense that the limit above exists) up to a set of $p$-capacity zero. Here the $p$-capacity of a set $E\subset\R^d$, that we denote by $\cp_p$, is defined by
$$\cp_p(E)=\inf\bigg\{\int\big(|\nabla u|^p+|u|^p\big)\,dx\ :\ u\in W^{1,p}_0(\R^d),\ u=1\text{ in a neighborhood of }E\bigg\}.$$
In the following, with an abuse of notation, we often identify the functions $u$ and $\tilde u$, since the sets of $p$-capacity zero do not influence in any way the quantities involved in our framework. 

\begin{df}\label{def:quasio}
Let $p>1$. A set $\O\subset\R^d$ is said to be $p$-quasi open if there exists a function $u\in W^{1,p}(\R^d)$ such that $\O=\{\tilde u>0\}$ up to sets of $p$-capacity zero.
\end{df}

Thanks to the Sobolev embedding theorem, when $p>d$, any $p$-quasi open set is necessarily open. Consequently, in this regime, the notions of open and $p$-quasi open sets coincide. In this case, any point has a positive capacity and $p$-quasi-continuity comes to continuity.

If $\O$ is $p$-quasi open, one can define the Sobolev space $W^{1,p}_0(\O)$ as the collection of all functions $u\in W^{1,p}(\R^d)$ such that $\tilde u=0$ {\it quasi-erverywhere} in $\R^d\setminus\O$. In particular, the first Dirichlet eigenvalue $\lambda_p(\O)$ is well-defined for every $p$-quasi open set $\O$, and can be characterized variationally by the expression given in \eqref{lp}.

An essential ingredient in the analysis of existence results for optimal domains associated with the shape functional $\F_{p,q}$ is the notion of $p$-capacitary measures and the related concept of $\gamma_p$-convergence. These tools allow us to extend the needed variational principles to a broader class of objects beyond classical open sets.

\begin{df}\label{def:capp}
Let $1<p\le d$. A nonnegative Borel measure $\mu$, possibly taking the value $+\infty$, is said to be of $p$-capacitary type if
$$\mu(E)=0\qquad\text{for every Borel set $E$ with $p$-capacity zero.}$$
\end{df}

Measures of $p$-capacitary type provide a natural generalization of $p$-quasi open sets. Indeed, if $\O\subset\R^d$ is a $p$-quasi open set, we may associate to it the Borel measure
$$\infty_{\R^d\setminus\O}(E)=\begin{cases}
0&\text{if }\cp_p(E\setminus\O)=0\\
+\infty&\text{otherwise,}
\end{cases}$$
which is easily seen to be of $p$-capacitary type.

More generally, given a $p$-capacitary measure $\mu$, one can define the first eigenvalue $\lambda_p(\mu)$ via the relaxed Rayleigh quotient:
$$\lambda_p(\mu)=\inf\bigg\{\frac{\int|\nabla u|^p\,dx+\int|u|^pd\mu}{\int|u|^p\,dx}\ :\ u\in W^{1,p}(\R^d)\setminus\{0\}\bigg\}.$$
In particular, when $\mu=\infty_{\R^d\setminus\O}$ for a $p$-quasi open set $\O$, we recover the usual Dirichlet eigenvalue:
$$\lambda_p(\infty_{\R^d\setminus\O})=\lambda_p(\O).$$
An immediate consequence of the variational formulation is the following monotonicity property:
\be\label{monot}
\lambda_p(\mu_1)\le\lambda_p(\mu_2)\qquad\text{whenever }\mu_1\le\mu_2\text{ (as measures)}.
\ee
When dealing with shape optimization problems under domain constraints, it is natural to restrict attention to subsets of a fixed bounded open set $D\subset\R^d$. In this context, the class of admissible domains is given by
$$\A(D)=\big\{\O\subset D\ :\ \O\text{ $p$-quasi open}\big\}.$$

\begin{df}\label{def:gammap}
Let $(\mu_n)$ be a sequence of $p$-capacitary measures supoorted in $D$ ($p>1$) and let $\mu$ be another $p$-capacitary measure on $D$. We say that $\mu_n$ $\gamma_p$-converges to $\mu$, and write $\mu_n\stackrel{\gamma_p}\to\mu$, if the corresponding sequence of functionals
$$F_n(u)=\int_D|\nabla u|^pdx+\int_D|u|^pd\mu_n\qquad u\in W^{1,p}_0(D)$$
$\Gamma$-converges in $L^p(D)$ to the functional
$$F(u)=\int_D|\nabla u|^pdx+\int_D|u|^pd\mu\qquad u\in W^{1,p}_0(D).$$
\end{df}

For a comprehensive treatment of $\Gamma$-convergence and its applications, we refer the reader to the classical monograph \cite{DM93}. A detailed discussion of shape optimization problems, the theory of $p$-quasi open sets, and the framework of $p$-capacitary measures can be found in \cite{BB05} and the references therein.

In the present setting, the key facts that will be relevant to our analysis are the following.
We summarize below some fundamental properties of $\gamma_p$-convergence which are essential in the study of variational problems involving $p$-capacitary measures.

\begin{itemize}
\item\textbf{Simplification in the supercritical case.} When $p>d$, the notion of $\gamma_p$-convergence for a sequence $(\O_n)$ of open sets reduces to the Hausdorff convergence of the corresponding complements (closed sets) $\overline{D}\setminus\O_n$ in the ambient domain. In this regime, the distinction between open and $p$-quasi open sets disappears due to the Sobolev embedding.

\item\textbf{Compactness.} The space of $p$-capacitary measures supported in $D$ is compact with respect to the $\gamma_p$-convergence. More precisely, any sequence $(\mu_n)$ of $p$-capacitary measures admits a $\gamma_p$-convergent subsequence, whose limit is again a $p$-capacitary measure.

\item\textbf{Metrizability and PDE characterization.} The space of $p$-capacitary measures supported in $D$ endowed with the $\gamma_p$-convergence is metrizable. Moreover, the convergence $\mu_n\to\mu$ in the $\gamma_p$-sense is equivalent to the strong convergence in $L^p(D)$ of the corresponding solutions $w(\mu_n)\to w(\mu)$, where each $w(\mu_n)$ solves the nonlinear PDE:
\be\label{wmu}
\begin{cases}
-\Delta_pw+\mu_n w^{p-1}=1&\text{in }D\\
w\in W^{1,p}_0(D),\quad w\ge0.
\end{cases}\ee
This correspondence allows one to define a natural metric, the so-called $\gamma_p$-distance:
$$d_{\gamma_p}(\mu,\nu)=\|w(\mu)-w(\nu)\|_{L^p}.$$

\item\textbf{Continuity of the eigenvalue.} The first eigenvalue $\lambda_p(\mu)$ depends continuously on the capacitary measure $\mu$ with respect to the $\gamma_p$-convergence.

\item\textbf{Characterization of $\gamma_p$ limits.} When $p\le d$, the class of all $\gamma_p$-limits of sequences of (quasi) open sets coincides precisely with the full space of $p$-capacitary measures. On the other hand, for $p>d$, the situation is more rigid: the class of open sets whose complements converge in the Hausdorff sense is already compact, and hence captures all possible $\gamma_p$-limits.

The first example illustrating the general behavior, in the subcritical regime $p\le d$ is the construction in \cite{CM97}, where a sequence $(\O_n)$ was shown to $\gamma_2$-converge to the Lebesgue measure. The full characterization of $\gamma_2$-limits in terms of capacitary measures was established in \cite{DMM87}, and the general result for all $p\le d$ was later obtained in \cite{DMD88}.
\end{itemize}

\begin{df}\label{def:omu}
Given a $p$-capacitary measure $\mu$ on $D$, we define the associated set $\O_\mu$ as the region where $\mu$ is finite. More precisely,
$$\O_\mu=\{w(\mu)>0\},$$
where $w(\mu)\in W^{1,p}_0(D)$ denotes the unique weak solution to the nonlinear elliptic problem
$$\begin{cases}
-\Delta_p w+\mu w^{p-1}=1&\text{in }D\\
w\ge0,\quad w\in W^{1,p}_0(D).
\end{cases}$$
\end{df}

The set $\O_\mu$ is $p$-quasi open by construction and plays a central role in the analysis of limit configurations arising in variational problems involving $p$-capacitary measures.

A fundamental result, which is needed to prove the existence of an optimal domain $\O_{opt}$ for the shape functional $\F_{p,q}$ introduced in \eqref{fpq}, is the following compatibility condition between the $\gamma_p$- and $\gamma_q$-limits of a sequence of domains.

\begin{trm}\label{th:ofin}
Let $1<q<p <+\infty$ and let $(\O_n)\subset\A(D)$ be a sequence of admissible domains such that
$$\O_n\stackrel{\gamma_p}\to\mu\qquad\text{and}\qquad\O_n\stackrel{\gamma_q}\to\nu.$$
for some $p$- and $q$-capacitary measures $\mu$ and $\nu$, respectively. Then, the measure $\nu$ vanishes on the $p$-quasi open set $\O_\mu$ where $\mu$ is finite.
\end{trm}

\begin{proof}
Let $w$ be the solution of \eqref{wmu} associated with the measure $\mu$, and let $(w_n)$ be an optimal sequence for $w$ with respect to the $\Gamma$-convergence. That is, $w_n\in W_0^{1,p}(\O_n)$ and
$$\lim_n\int|\nabla w_n|^pdx=\int|\nabla w|^pdx+\int w^pd\mu.$$
Fix now a parameter $\alpha>0$ and define the auxiliary function
$$\varphi(x)=\big(w(x)-\alpha\big)^+,$$
where, as usual, $x^+$ denotes the positive part of $x$. In addition, we introduce the Lipschitz continuous function $H:\R\to[0,1]$ given by
$$H(s)=(s/\alpha)\wedge1.$$
We then consider the sequence $(u_n)$ defined by
$$u_n=\varphi\,H(w_n).$$
Since $\varphi\in W_0^{1,p}(\O)$ and $H(w_n)$ is bounded, it follows that $u_n\in W^{1,p}_0(\O_n)$ for each $n$. Moreover, by the strong convergence $w_n\to w$ in $L^p$ and the continuity of $H$, it is easy to see that $u_n\to u=\varphi\,H(w)$ strongly in $L^q$. Applying the $\Gamma$-liminf inequality to the sequence $(u_n)$ we obtain
$$\liminf_n\int|\nabla u_n|^qdx\ge\int|\nabla u|^qdx+\int u^qd\nu.$$
Recalling that $u=\varphi H(w)=(w-\alpha)^+$ yields
$$\int|\nabla u|^qdx+\int u^qd\nu=\int|\nabla\varphi|^qdx+\int\varphi^qd\nu,$$
so that
\[\begin{split}
\int\big((w-\alpha)^+\big)^qd\nu&\le-\int|\nabla\varphi|^qdx+\liminf_n\int|\nabla u_n|^qdx\\
&\le\limsup_n\int\Big[\big|H(w_n)\nabla\varphi+\varphi H'(w_n)\nabla w_n\big|^q-|\nabla\varphi|^q\Big]dx.
\end{split}\]
In order to estimate the right-hand side, we use the standard inequality
$$|b|^q-|a|^q\le C|b-a|\big(|a|^{q-1}+|b|^{q-1}\big),$$
which holds for some constant $C>0$ depending only on $q$. Applying this inequality with $a=\nabla\varphi$ and $b=H(w_n)\nabla\varphi+\varphi H'(w_n)\nabla w_n$, we obtain
\[\begin{split}
\int\big((w-\alpha)^+\big)^qd\nu&\le C\limsup_n\int\Big[\big|H(w_n)-1\big||\nabla\varphi|+\big|\varphi H'(w_n)\nabla w_n\big|\Big]\\
&\qquad\cdot\Big[|\nabla\varphi|^{q-1}+\big|\varphi H'(w_n)\nabla w_n\big|^{q-1}\Big]\,dx.
\end{split}\]
Applying H\"older's inequality, we estimate
\[\begin{split}
&\int\Big[\big|H(w_n)-1\big||\nabla\varphi|+\big|\varphi H'(w_n)\nabla w_n\big|\Big]\\
&\qquad\cdot\Big[|\nabla\varphi|^{q-1}+\big|\varphi H'(w_n)\nabla w_n\big|^{q-1}\Big]\,dx.\\
\le& C\bigg[\int|\nabla\varphi|^q\big|H(w_n)-1\big|^q+\big|\varphi H'(w_n)\nabla w_n\big|^qdx\bigg]^{1/q}\\
&\qquad\cdot\bigg[\int|\nabla\varphi|^q+\big|\varphi H'(w_n)\nabla w_n\big|^qdx\bigg]^{(q-1)/q}.
\end{split}\]
for some constant $C>0$ depending only on $q$. Next, we apply H\"older's inequality once more to the term involving $\varphi H'(w_n)\nabla w_n$, obtaining
$$\int\big|\varphi H'(w_n)\nabla w_n\big|^qdx\le\bigg[\int|\nabla w_n|^p\bigg]^{q/p}\bigg[\int\big|\varphi H'(w_n)\big|^{p/(p-q)}\bigg]^{(p-q)/p}.$$
Since $(w_n)$ is an optimal sequence for $w$ in the $\Gamma$-convergence, it follows that as $n\to\infty$, the right-hand side above tends to
$$\bigg[\int|\nabla w|^pdx+\int w^pd\mu\bigg]^{q/p}\bigg[\int\big|\varphi H'(w)\big|^{p/(p-q)}\bigg]^{(p-q)/p},$$
which vanishes, since $\varphi H'(w)=0$. Similarly, considering the term
$$\int|\nabla\varphi|^q\big|H(w_n)-1\big|^qdx,$$
we observe that $H(w_n)\to H(w)$ almost everywhere, and by the definition of $\varphi$ and $H$, we have $H(w)=1$ wherever $\varphi$ is nonzero. Thus, the dominated convergence theorem applies, yielding
$$\lim_{n\to\infty}\int|\nabla\varphi|^q|H(w_n)-1|^qdx=0.$$
Combining all these facts and recalling the previous estimates, we conclude that
$$\int\big((w-\alpha)^+\big)^qd\nu=0.$$
We deduce that $(w-\alpha)^+=0$ $\nu$-almost everywhere and, since $\alpha>0$ was chosen arbitrarily, we deduce that $w=0$ $\nu$-almost everywhere, then concluding the proof.
\end{proof}

\section{Existence of optimizers in $\Rd$}\label{sexis}

In this section we give a proof of existence of optimal domains for the shape functional $\Fpq$, with $1\le q\le p\le+\infty$, in the entire space $\R^d$.

\subsection{Proof of Theorem \ref{Th:ExistenceMpqmpq}}

\begin{proof} We prove only the existence for the maximization problem, the proof being the same in the minimization case: the only difference is to fix $\lambda_q=1$ instead of $\lambda_p=1$.

Fix $p>q$, and let $(\O_n)_n$ be a minimizing sequence of bounded open sets, which can also be taken smooth. By the scaling invariance of $\Fpq$, one may assume that $\lambda_p(\O_n)=1$ for every $n$. Since every $\O_n$ is bounded, one can find $R_n>0 $ such that $\O_n\subset B(0,R_n)$, and fixing a unit vector $\mathbf{e}$ of $\Rd$ we define recursively the sequence $(x_n)_n$ of vectors by
$$\begin{cases}
x_{n+1}=x_n+(R_n+R_{n+1}+1)\,\mathbf{e}\\
x_1=0.
\end{cases}$$
The set
$$\tilde\O=\bigcup_{n\geq 1} \big(x_n + \O_n\big)$$
is still an open (and smooth) subset of $\Rd$ and moreover, since the sets $(x_n+\O_n)_n$ are disjoint, we have
$$\lambda_p(\tilde \O)=\inf_n\lambda_p(\O_n)=1,\qquad\lambda_q(\tilde\O)=\inf_n\lambda_q(\O_n)=M(p,q)^{-q}.$$
Therefore, $\tilde\O$ is a maximizer of $\Fpq$.

It only remains to prove the second claim. If $q>d$, we know from Theorem \ref{th:BozzolaBrasco24} that for any open subset $\O$ of $\Rd$, we have
$$\lambda_p(\O)>0\iff\rho(\O)<+\infty\iff\lambda_q(\O)>0.$$
In particular, $\lambda_p(\tilde\O)=1$ implies $\lambda_q(\tilde\O)>0$ which yields $M(p,q)<+\infty$.

If $q\le d<p$, since $\Rd\setminus\Z^d$ has a finite inradius, we have $\lambda_p(\Rd\setminus\Z^d)>0$ and yet, $\Z^d$ has zero $q$-capacity, so that $ \lambda_q(\Rd\setminus\xi\,\Z^d)=\lambda_q(\Rd)=0$, which gives $M(p,q)=+\infty$.

Finally, let us consider the case $q<p\le d$. By following the construction for $p=2$ of Cioranescu and Murat \cite{CM97}, generalized in \cite{DMD88} to any $p\le d$, one can find a sequence of radii $(r_n)$ such that the sequence of perforated balls $(B_n)$ defined by
$$B_n=B(0,1)\setminus\Big(\frac1n\Z^d\,\oplus\,B(0,r_n)\Big)$$
satifies $\lambda_p(B_n)\to+\infty$, while $\lambda_q(B_n)\to C$ for some constant $C>0$. Consequently, $\Fpq(B_n)\to+\infty$ and so $M(p,q)=+\infty$.
\end{proof}

\subsection{The nonlinearity gap}

We consider here the shape functional
\be\label{Jpq}
J_{p,q}(\O)=\lambda_p(\O)-\lambda_q(\O)\qquad\text{with }p>q,
\ee
that we call {\it``nonlinear gap''}, and the related shape optimization problem
\be\label{ng1}
\min\big\{J_{p,q}(\O)\ :\ \O\subset\R^d\big\}.
\ee
Notice that, by taking $\O=tB$ with $B$ a ball of unitary radius and $t\to0$, the supremum of the shape functional $J_{p,q}$ is $+\infty$.

\begin{trm}\label{thng1}
Let $1\le q\le p\le+\infty$. The shape optimization problem \eqref{ng1} admits a smooth unbounded solution $\O_0$ that is homothetic to the minimizer of the functional $\Fpq$ obtained in Theorem \ref{Th:ExistenceMpqmpq}. More precisely, the problems \eqref{ng1} and \eqref{df:mpq,Mpq} (for $m(p,q)$) have the same optimal domains, up to homotheties.
\end{trm}

\begin{proof}
If in problem \eqref{ng1} we consider $t\O$ instead of $\O$ we have
$$\min_{\O\subset\R^d}\min_{t>0}\Big\{t^{-p}\lambda_p(\O)-t^{-q}\lambda_q(\O)\Big\}.$$
The optimization with respect to $t$ can be easily computed and we obtain
$$\min_{\O\subset\R^d}-\frac{p-q}{p}\Big(\frac{q}{p}\Big)^{q/(p-q)}\Big(\frac{\lambda_q^{1/q}(\O)}{\lambda_p^{1/p}(\O)}\Big)^{pq/(p-q)},$$
so that the problem is reduced to the minimization of $\Fpq$. If $\O_{opt}$ is a solution for $\F_{p,q}$ we then have
$$\O_0=\overline t\,\O_{opt}\qquad\text{with }\overline t=\Big(\frac{p\lambda_p(\O_{opt})}{q\lambda_q(\O_{opt})}\Big)^{1/(p-q)},$$
which concludes the proof.
\end{proof}

\begin{rem}\label{rem:generalphiRd}
By following the same arguments as above we actually obtain that the generalized problem 
$$\min\Big\{\Phi\big(\lambda_p(\O)^{\frac{1}{p}},\lambda_q(\O)^\frac{1}{q}\big)\ :\ \O\subset\Rd,\ \text{open}\Big\}$$
admits a solution and its minimizers are homotetic to the minimizers of $\Fpq$, for any function $\Phi:\R^*_+\times\R^*_+\to\R$ such that 
\begin{enumerate}
\item[(i)] $\Phi(x,y)$ is either increasing in $x$ or decreasing in $y$;
\item[(ii)] For any $x > 0$, there exists $s_x^* >0$ such that $\inf_{s>0}\Phi(s\,x,\,s) = \Phi (s_x^*\,x,\,s_x^*)$.
\end{enumerate}
This class of function contains the function 
$$
f(x,y) = \frac{x}{y},
$$
which corresponds to our main functionnal $\Fpq$, it also contains the family
$$
j_{\alpha,\beta} (x,y)  = x^\alpha - y^\beta,
$$
for any $\alpha> \beta >0$ where the nonlinearity gap $J_{p,q}$ corresponds to $j_{p,q}$.
\end{rem}

\section{Description of the optimizers}
We now give a proof of Theorem \ref{th:discret}

\begin{proof}We prove the result for the minimization problem, the proof being the same for the maximization case.

Take $(\O_n)_n$ a minimizing sequence of bounded sets for $\Fpq$ such that $\lambda_q(\O_n)=1$ for all $n$, and assume without loss of generality that the sequence $(\lambda_p(\O_n))_n$ is decreasing. We set for each $n$, $\tilde{\O}_n =(1+\frac{1}{n})^{-1/q}\O_n $; since the $\tilde{\O}_n$ are bounded, for each $n$ there exists $R_n>0$ such that $\O_n\subset B(0,R_n)$. We choose a sequence $(x_n)_n$ of elements of $\Rd$ such that the distance between the $(B(x_n,R_n))_n$ is at least $1$.

Consider now $\xi_0>0$ such that
$$
\lambda_q(\Rd\setminus \xi_0\Z^d) >2\qquad\text{and}\qquad\lambda_p(\Rd\setminus \xi_0\Z^d)> \lambda_p(\O_1).
$$
We define the set 
$$A=\big(\Rd\setminus\xi_0\Z^d \big) \setminus \bigcup_{n\ge1} \overline B(x_n, R_n+1/4),$$
and we construct the sequence of sets $(A_n)_n$ such that 
$$
A_n=A\cup\bigcup_{1\leq k\leq n} C_k,\qquad\text{with}\quad C_k = x_k +\ \big(\overline{B}(0,\ R_k+1/4) \setminus \xi_k\Z^d\big) \cup \tilde{\O}_k,
$$
where $(\xi_n)_n$ are suitable real numbers. We claim that we can choose the real numbers $\xi_n$ such that
\be\label{claim}
1+\frac{1}{n+1} < \lambda_q(A_n) \leq 1+\frac{1}{n}.
\ee
Then, since the sequence $(A_n)_n$ is increasing for the inclusion of sets, it converges for both the $\gamma_p$ and $\gamma_q$ convergences to 
$$
\O_m = \bigcup_{n\geq 1 } A_n = A\cup \bigcup_{n\geq 1} C_n.
$$
By $\gamma_q$-convergence, $\lambda_q(\O_m) = \lim \lambda_q(A_n) = 1$, and from the inclusion of sets 
$$
\lambda_p(\O_m) \leq \inf_n \lambda_p(\tilde{\O}_n) = \inf_n \big(1+\frac{1}{n}\big)^{p/q} \lambda_p(\O_n) = m(p,q)^p,$$
which means that $\O_m$ is optimal and of discrete complement.

It only remains to prove the claim \eqref{claim}. By inclusion the right inequality is of course satisfied for any $n$, since
$$\lambda_q(A_n)\le\lambda_q(\tilde{\O}_n) = 1+\frac{1}{n}.$$
As for the left inequality, it comes from the fact that for any open set $\O$ in $\Rd$ containing the unit ball $B(0,1)$ and any open subset $\omega$ of the ball of radius $3/4$, the sequence of open sets $(C_\xi)_{\xi>0}$ defined by 
$$C_\xi = \O\setminus \big(B(0,1)\cap \xi\Z^d\big) \cup \omega,$$ 
$\gamma_q$-converges to $\O\setminus \overline{B}(0,1) \cup \omega$ as $\xi$ tends to $0$, since $q>d$.
\end{proof}

\section{The bounded case}\label{s:DDD}

In this section, we study the problem in which the admissible domains $\O$ are contained in a given open bounded subset $D$ of $\Rd$. We consider the optimization problems \eqref{pb:boundedmin} and \eqref{pb:boundedmax}. Of course, in order to have nontrivial problems, we require that $\lambda_q(D)<1$ in \eqref{pb:boundedmin} and $\lambda_p(D)<1$ in \eqref{pb:boundedmax}.

\subsection{Proof of Theorem \ref{quasi}}

\begin{proof}
Let us consider first the minimization problem \eqref{pb:boundedmin} and let $(\O_n)$ be a minimizing sequence. Since the $\gamma_p$ and $\gamma_q$ convergences are compact, possibly passing to subsequences, we may also assume that
$$\O_n\stackrel{\gamma_p}\to\mu\qquad\text{and}\qquad\O_n\stackrel{\gamma_q}\to\nu,$$
for some $p$-capacitary measure $\mu$ and $q$-capacitary measure $\nu$. If $p>d$, since the $\gamma_p$ convergence reduces to the Hausdorff convergence of the complements $\overline D\setminus\O_n$, the measure $\mu$ will be of the form $\mu=\infty_{\overline D\setminus\O}$ for a suitable open set $\O$. By Theorem \ref{th:ofin} we have $\nu=0$ on $\O$ and, by the continuity of $\lambda_p$ and $\lambda_q$ with respect to the $\gamma_p$ and $\gamma_q$ convergences respectively, we have
$$\begin{cases}
\lambda_p(\O)=\lim_n\lambda_p(\O_n)=\inf\eqref{pb:boundedmin},\\
\lambda_q(\O)\ge\lambda_q(\nu)=\lim_n\lambda_q(\O_n)=1.
\end{cases}$$
Since $\lambda_q(D)<1$ we may take $\O'$ such that $\O\subset\O'\subset D$ and $\lambda_q(\O')=1$. We have $\lambda_p(\O')\le\lambda_p(\O)$ and so $\O'$ is optimal.

We now consider the case $q\le d$. Again, by the continuity of $\lambda_p$ and $\lambda_q$ with respect to the $\gamma_p$ and $\gamma_q$ convergences respectively, we have
$$\begin{cases}
\lambda_p(\mu)=\lim_n\lambda_p(\O_n)=\inf\eqref{pb:boundedmin},\\
\lambda_q(\nu)=\lim_n\lambda_q(\O_n)=1.
\end{cases}$$
By Theorem \ref{th:ofin} we have that $\nu=0$ on the set $\O_\mu$ where $\mu$ is finite; therefore, from the monotonicity property \eqref{monot} we deduce that
$$\begin{cases}
\lambda_p(\O_\mu)\le\lambda_p(\mu)=\inf\eqref{pb:boundedmin},\\
\lambda_q(\O_\mu)\ge\lambda_q(\nu)=1.
\end{cases}$$
As above, we may take $\O'$ such that $\O_\mu\subset\O'\subset D$ and $\lambda_q(\O')=1$. We have $\lambda_p(\O')\le\lambda_p(\O)$, which gives the optimality of $\O'$.

The existence of an optimal domain for the maximization problem \eqref{pb:boundedmax} can be obtained in a similar way by repeating the argument above.
\end{proof}

We now prove that if one could prove the existence of solutions for the bounded scale invariant problems 
$$
\begin{aligned}
m(p,q,D) &= \inf\big\{ \Fpq (\O) \ : \ \O\in D \big\},\\
M(p,q,D) &= \sup\big\{ \Fpq (\O) \ : \ \O\in D \big\},
\end{aligned}
$$
these solutions would actually be global solutions.

\begin{trm}\label{th:inf=inf}
For any $p<q$, and any $D$ open subset of $\Rd$
$$m(p,q,D)=m(p,q)\qquad\text{and}\qquad M(p,q,D)=M(p,q)$$
\end{trm}

\begin{proof}
We only prove the equality $m(p,q, D) = m(p,q)$, the proof of the equality $M(p,q, D) = M(p,q)$ being the same, the only difference is to fix $\lambda_p$ instead of $\lambda_q$.\\
Consider $p>q$ and $D\subset \Rd$ open and bounded. Take $(\O_n)_n$ a minimizing sequence of $\Fpq$ of open bounded subsets of $\Rd$, from the scale invariance, one may assume that for all $n$, $\lambda_q(\O_n) = 1$.\\
Since for all $n$, $\O_n$ is bounded one can find $t_n$ and $x_n$ such that $x_n + t_n\, \O_n \subset D$. Therefore, since $\Fpq$ is invariant for translations and scalings, one has
$$
\forall n,\qquad m(p,q,D)\leq \Fpq(x_n + t_n\, \O_n) = \Fpq(\O_n).
$$
Finally passing to the limit, we obtain, $m(p,q,D)\leq m(p,q)$, the other inequality being trivial this proves the theorem.
\end{proof}

\subsection{The constrained nonlinearity gap}

We consider here the optimization problem for the shape functional $J_{p,q}$ in \eqref{Jpq} subject to the constraint $\O\subset D$:
\be\label{ng2}
\min\big\{J_{p,q}(\O)\ :\ \O\subset D\big\}.
\ee

\begin{trm}\label{thng2}
The shape optimization problem \eqref{ng2} admits a solution $\O^*$ that is a $p$-quasi open set.
\end{trm}

\begin{proof}
Let $(\O_n)$ be a minimizing sequence; by the compactness of $\gamma_p$ and $\gamma_q$ convergences we may assume that
$$\O_n\stackrel{\gamma_p}\to\mu\qquad\text{and}\qquad\O_n\stackrel{\gamma_q}\to\nu,$$
for some $p$-capacitary measure $\mu$ and $q$-capacitary measure $\nu$. Then the infimum in \eqref{ng2} is given by
$$\lambda_p(\mu)-\lambda_q(\nu).$$
Since for some $\O_0\subset D$ we have
$$J_{p,q}(\O_n)\le J_{p,q}(\O_0)$$
and
$$\F_{p,q}(\O_n)\ge q/p,$$
we find that the quantities $\lambda_p(\O_n)$ and $\lambda_q(\O_n)$ are both bounded. By Theorem \ref{th:ofin} we have that $\nu=0$ on the $p$-quasi open set $\O_\mu$ where $\mu$ is finite. Then
$$\lambda_p(\O_\mu)\le\lambda_p(\mu)\qquad\text{and}\qquad\lambda_q(\O_\mu)\ge\lambda_q(\nu),$$
which implies that $\O_\mu$ is optimal for the constrained shape functional $J_{p,q}$.
\end{proof}

\begin{rem}
By repeating step by step the proof of Theorem \ref{thng2} we obtain the existence of an optimal $p$-quasi open set $\O_{opt}$ for the shape optimization problem
$$\min\Big\{\Phi\big(\lambda_p(\O),\lambda_q(\O)\big)\ :\ \O\subset D\Big\}$$
whenever $\Phi:\R^+\times\R^+\to\R$ is a function such that:
\begin{itemize}
\item[(i)]$\Phi$ is continuous;
 \item[(ii)]$\Phi(x,y)$ is increasing in $x$ and decreasing in $y$;
\item[(iii)]$\lim_{t\to+\infty}\Phi(t^p,t^q)=+\infty$.
\end{itemize}
\end{rem}

\section{further remarks and open problems}\label{s:open}

In this section we summarize some questions that in our opinion should be considered.

\subsection{Questions for $M_{p,q}$}As noticed in Theorem \ref{Th:ExistenceMpqmpq} the only interesting case is $p>q>d$, since $M(p,q)=+\infty$ when $q\le d$.

\medskip\noindent{\it Question Q1}. The characterization of the maximal value $M_{p,q}$ is an interesting issue.

\medskip\noindent{\it Question Q2}. In Theorem \ref{Th:ExistenceMpqmpq} we have shown the existence of optimal unbounded domains $\O_{p,q}^M$ for which $\Fpq(\O_{p,q}^M)=M(p,q)$. Is it possible to have an optimal bounded domain? We believe this is not possible.

\medskip\noindent{\it Question Q3}. In Theorem \ref{th:discret} we have shown the existence of optimal domains $\O_{p,q}^M$ such that $\Rd\setminus\O_{p,q}^M$ is a discrete set. It would be interesting to show the existence of optimal domains $\O_{p,q}^M$ such that $\Rd\setminus\O_{p,q}^M$ is a {\it periodic} discrete set.

\medskip\noindent{\it Question Q4}. In the planar case $d=2$ it would be interesting to prove (or disprove) that the domain $\R^2\setminus X$ is optimal, where $X$ is the discrete periodic set consisting of all the centers of a planar tessellation made of regular hexagons. We believe this to be true, at least in the extreme case $p=\infty$.

\subsection{Questions for $m_{p,q}$}In this case Theorem \ref{Th:ExistenceMpqmpq} gives the existence of optimal unbounded domains $\O_{p,q}^m$ for every $q\le p$, and the inequality $m_{p,q}\ge q/p$ holds.

\medskip\noindent{\it Question Q5}. Also in this case the characterization of the minimal value $m_{p,q}$ is an interesting issue.

\medskip\noindent{\it Question Q6}. It would be interesting to see in which cases a {\it bounded} optimal domain $\O_{p,q}^m$ exists.

\medskip\noindent{\it Question Q7}. If $p=\infty$ it is easy to see that a ball is an optimal domain for every $q$. Is there a threshold $p^*$ (possibly depending on $q$) such that for $p\ge p^*$ a {ball} is still optimal?

\subsection{Questions for the problem with a bounded constraint}In Theorem \ref{quasi} we proved the existence of an optimal $p$-quasi open sets $\O_{opt}^M$ and $\O_{opt}^m$ solving respectively the maximization and the minimization problems \ref{pb:boundedmax} and \ref{pb:boundedmin}.

\medskip\noindent{\it Question Q8}. It would be interesting to study the regularity of $\O_{opt}^M$ and $\O_{opt}^m$ in various steps:
\begin{itemize}
\item[-] see if $\O_{opt}^M$ and $\O_{opt}^m$  are {open} sets;\\
\item[-] see if $\O_{opt}^M$ and $\O_{opt}^m$  are sets with finite perimeter;\\
\item[-] see if $\O_{opt}^M$ and $\O_{opt}^m$  have a higher regularity, for instance $C^{1,\alpha}$, possibly up to a small singular set.
\end{itemize}

\medskip\noindent{\it Question Q9}. In the case with a bounded constraint $D$ the maximization problem \eqref{pb:boundedmax} is slightly different from the minimization problem for $\Fpq$. We do not know if this last one admits an optimal solution.

\medskip\noindent{\it Question Q10}. We proved in Theorem \ref{thng2} existence of a $p$-quasi open solution to the constraint nonlinearity gap problem 
$$\min\big\{\lambda_p(\O)-\lambda_q(\O)\ :\ \O\subset D\big\},$$
and to its generalisation
$$\min\big\{\Phi\big(\lambda_p(\O),\lambda_q(\O)\big)\ :\ \O\subset D\big\},$$
where $\Phi(x,y)$ is a continuous function, increasing with respect to $x$ and decreasing with respect to $y$. The points listed in {\it Question 8} should be investigated in this more general framework.\\
We also showed in Remark \ref{rem:generalphiRd} that the minizers for the problem 
$$I = \min\big\{\Phi\big(\lambda_p(\O),\lambda_q(\O)\big)\ :\ \O\subset \Rd\big\},$$
are the same as the minimizers of $\Fpq$ in $\Rd$ up to a rescalling, and therefore, it is possible to construct a minimzing sequence of bounded sets $(\O_n)$ solving the problem $I$. Then, the sequence $(I_n)$ defined as 
$$
\forall n, \qquad I_n = \min\big\{\Phi\big(\lambda_p(\O),\lambda_q(\O)\big)\ :\ \O\subset B(0,n)\big\},
$$
converges to $I$. It would be of interest to understand whether the associated sequence of solution $(\O^*_n)$ converges to a global minimizer $\O^*$ and to understand the properties of this selected minimizer.

\bigskip\bigskip

\noindent{\bf Acknowledgments.} The work of GB is part of the project 2017TEXA3H {\it``Gradient flows, Optimal Transport and Metric Measure Structures''} funded by the Italian Ministry of Research and University. GB is member of the Gruppo Nazionale per l'Analisi Matematica, la Probabilit\`a e le loro Applicazioni (GNAMPA) of the Istituto Nazionale di Alta Matematica (INdAM). The work of DB and AdV has been supported by ANR STOIQUES, ANR-24-CE40-2216.

\bigskip


\bigskip\noindent
Dorin Bucur: Laboratoire de Math\'ematiques (LAMA), Universit\'e de Savoie\\
Campus Scientifique - 73376 Le-Bourget-Du-Lac, FRANCE\\
{\tt dorin.bucur@univ-savoie.fr}\\
{\tt https://bucur.perso.math.cnrs.fr/}

\bigskip\noindent
Giuseppe Buttazzo: Dipartimento di Matematica, Universit\`a di Pisa\\
Largo B. Pontecorvo 5, 56127 Pisa - ITALY\\
{\tt giuseppe.buttazzo@unipi.it}\\
{\tt http://www.dm.unipi.it/pages/buttazzo/}

\bigskip\noindent
Alexis de Villeroch\'e: Laboratoire de Math\'ematiques (LAMA), Universit\'e de Savoie\\
Campus Scientifique - 73376 Le-Bourget-Du-Lac, FRANCE\\
{\tt alexis.devilleroche@univ-smb.fr}\\
{\tt https://alexisdevilleroche.perso.math.cnrs.fr/}

\end{document}